\documentclass[10pt,reqno]{article}
\usepackage{amsmath,amsfonts,amssymb, amsthm, enumerate, dsfont, hyperref,multicol,cite}
\usepackage{authblk,color,epstopdf}
\usepackage{blindtext}
\usepackage{textgreek}
\usepackage{upgreek}
\usepackage{amstext}
\theoremstyle{plain}
 \newtheorem{thm}{Theorem}[section]
 
 \newtheorem{lemma}{Lemma}[section]
 \newtheorem{cor}{Corollary}[section]
\theoremstyle{definition}

\theoremstyle{remark}
 
\numberwithin{equation}{section}

\title{Durrmeyer type operators linked with Boas-Buck type polynomials }
\author{Naokant Deo and Sandeep Kumar}
\affil{Department of Mathematics, Delhi Technological University\\
Delhi-110042, India\\
 naokantdeo@dce.ac.in}
 \affil{Department of Mathematics, Maitreyi College, University of Delhi\\ Delhi 110021, India\\
sandeepkumardps@gmail.com}
\date{}
\begin{document}


\maketitle
\begin{abstract}
The present article intends to introduce the sequence of Baskakov-Durrmeyer type operators linked with the generating functions of Boas-Buck type polynomials. After calculating the moments, including the limiting case of central moments of the constructed sequence of operators, in the subsequent sections, we estimate the convergence rate using the Ditzian-Totik modulus of smoothness and some convergence results in Lipchitz-type space.
\end{abstract}


\section{Introduction}

Suppose that $ \xi, \mathcal{S}, \mathcal{T}, \mathcal{U} $ and $\mathcal{V}$ are analytic functions defined on some disc $|z|<\rho $ with $\rho >1.$ Let the power series expansions of these analytic functions are
\begin{eqnarray*}
 \xi(s)=\sum_{j=0}^{\infty}p_{j}s^{j},\,\,\,\,\,\,\,\,\,\,\,\,\,  \mathcal{S}(s)=\sum_{j=0}^{\infty}q_{j}s^{j},\,\,\,\,\,\,\,\,\,\,\,\,\, \mathcal{T}(s)=\sum_{j=0}^{\infty}r_{j}s^{j+1},
 \end{eqnarray*}
 \begin{eqnarray*}
 \mathcal{U}(s)=\sum_{j=0}^{\infty}u_{j}s^{j+2},\,\,\,\,\,\,\,\,\,\,\,\,\, \mathcal{V}(s)=\sum_{j=0}^{\infty}v_{j}s^{j+3},
\end{eqnarray*}where the coefficients of the above series representation meets with $p_{j}\neq0, q_{j}\neq 0 $ and $ r_{j}\neq 0. $ Now, according to R. Goyal \cite{Ritu} and E. D. Rainville \cite{Rainville}, we have polynomials generating function $\Theta_{j}(x)s^{j}$ as:

\begin{eqnarray*}
\mathcal{S}(s).\xi\left(x^2\mathcal{T}(s)+x \mathcal{U}(s)+ \mathcal{V}(s) \right)=\sum_{j=0}^{\infty}\Theta_j(x)s^j.
\end{eqnarray*}
Using the polynomial generating function $\Theta_j$, S. Verma and S. Sucu \cite{sucu:2} described the following operators
\begin{eqnarray}\label{eq1.1}
\mathcal{B}_n\left(f;x\right)=\frac{1}{\mathcal{S}(1).\xi\left(n^2x^2\mathcal{T}(1)+ nx \mathcal{U}(1)+\mathcal{V}(1)\right)}\sum_{j=0}^{\infty}\Theta_j(nx)f\left(\frac{j}{n}\right),
\end{eqnarray}with some restrictions for the positivity like $\xi(s)\geq0,\,\,\, \forall\,\ s \in \mathbb{R}$, and $\Theta_{j}(x)\geq 0; \,\ j=0,1,2 \ldots $ for all $x\geq0$ and $\mathcal{S}(1)>0$. In the reference of convergence of the sequence (\ref{eq1.1}) we assume that, $\mathcal{T}^{\prime}(1)=0, \mathcal{T}^{\prime\prime}(1)=0 $ and $\mathcal{U}^{\prime}(1)=1$ and apply the Korovkin-type theorem.
Many authors \cite{sucu:1}, \cite{Boas:1}, \cite{Boas:2}, \cite{Rasa}, \cite{gupta:1}, \cite{SANDEO:1} and \cite{Ana:1}  discussed the convergence of the related operators.
 The remarkable work of \cite{sucu:2}, \cite{PNA:1} and
\cite{SP1985} allow us to define a sequence of Baskakov-Durrmeyer type operators linked with the generating function of Boas-Buck type polynomials under the same assumptions and restrictions. \\
For $\sigma>0,$ consider a class of functions $C_{\sigma}\left[0,\infty\right)=\{f\in C\left[0,\infty\right): |f(s)|\leq M(1+s^\sigma)\}$, for some $M>0$ with the norm $\|f(s)\|=\displaystyle\sup_{s\in \left[0,\infty\right)}\frac{|f(s)|}{1+s^{\sigma}}.$
Now, for $f\in C_{\sigma}\left[0,\infty\right) $, we construct
\begin{eqnarray}\label{eq1.2}
\tilde{\mathcal{B}}_n\left(f;x\right)=\frac{1}{\mathcal{S}(1).\xi\left(n^2x^2\mathcal{T}(1)+ nx \mathcal{U}(1)+\mathcal{V}(1)\right)}\sum_{j=0}^{\infty}\frac{\Theta_j(nx)}{B(j,n+1)}\nonumber\\
\times\int_{0}^{\infty}\frac{s^{j-1}}{(1+s)^{n+j+1}}f(s)ds,
\end{eqnarray}where $B$ is a beta function, defined as $B(j,n+1)=\frac{\Gamma{(j)}\Gamma{(n+1)}}{\Gamma{(j+n+1)}}$.\\
Alternatively the operators (\ref{eq1.2}) may be written as:
\begin{eqnarray}\label{eq1.3}
\tilde{\mathcal{B}}_n\left(f;x\right)= \int_{0}^{\infty}\underline{\mathcal{V}}_{n}(x,s)(f;x)f(s)ds,
\end{eqnarray}where
\begin{eqnarray*}
\underline{\mathcal{V}}_{n}(x,s)=\frac{1}{\mathcal{S}(1).\xi\left(n^2x^2\mathcal{T}(1)+ nx \mathcal{U}(1)+\mathcal{V}(1)\right)}\sum_{j=0}^{\infty}\frac{\Theta_j(nx)}{B(j,n+1)}\frac{s^{j-1}}{(1+s)^{n+j+1}}.
\end{eqnarray*}
This article intends to investigate the convergence properties of the operators (\ref{eq1.2}).  Firstly, we calculate the moments and central moments, including the limiting case of the constructed                                                                                                                                                                                                                                                                                                                                                                                                                                                                                                                                                                                                                                                                                                                                                                                                                                                                                                                                                                                                                                                                                                                                                                                                                                                                                                                                                                                                              operators'  and then prove the convergence using Korovkin's theorem. After making some convergence results based on the modulus of continuity, Lipchitz-type space, and weighted space,  we estimates the convergence for the functions of bounded variations.
\section{Some results on the operators $ \mathcal{B}_n$ and $\tilde{\mathcal{B}}_n$}
\begin{lemma}\label{lemma2.1} From \cite{PNA:1},
for $x\in\left[0,\infty\right)$ and $p(x)=n^2x^2\mathcal{T}(1)+ nx \mathcal{U}(1)+\mathcal{V}(1)$ important results related to the operators (\ref{eq1.1}) are follows:
\begin{eqnarray*}
\mathcal{B}_n\left(1;x\right)&=&1\\
\mathcal{B}_n\left(s;x\right)&=&\frac{\xi^{\prime}\left(p(x)\right)}{\xi\left(p(x)\right)}x+\frac{1}{n}\left[\frac{\mathcal{S}^{\prime}(1)}{\mathcal{S}(1)}
+\mathcal{V}^{\prime}(1)\frac{\xi^{\prime}\left(p(x)\right)}{\xi\left(p(x)\right)}\right]\\
\mathcal{B}_n\left(s^2;x\right)&=& \frac{\xi^{\prime\prime}\left(p(x)\right)}{\xi \left(p(x)\right)}x^2 +\left[\frac{\xi^{\prime}\left(p(x)\right)}{\xi \left(p(x)\right)}\left(\frac{2\mathcal{S}^{\prime}(1)}{\mathcal{S}(1)}+\mathcal{U}^{\prime\prime}(1)+1\right)\right.\\
&&\left.
+2\mathcal{V}^{\prime}(1)\frac{\xi^{\prime\prime}\left(p(x)\right)}{\xi \left(p(x)\right)}\right]\frac{x}{n}+\left[\frac{\mathcal{S}^{\prime\prime}(1)+\mathcal{S}^{\prime}(1)}{\mathcal{S}(1)} +\left(\frac{2\mathcal{S}^{\prime}(1)\mathcal{V}^{\prime}(1)}{\mathcal{S}(1)}\right.\right.\\
&&\left.\left.+\mathcal{V}^{\prime\prime}(1)+\mathcal{V}^{\prime}(1)\right)
\frac{\xi^{\prime}\left(p(x)\right)}{\xi \left(p(x)\right)}+\left(\mathcal{V}^{\prime}(1)\right)^{2}\frac{\xi^{\prime\prime}\left(p(x)\right)}{\xi \left(p(x)\right)}  \right]\frac{1}{n^2}.
\end{eqnarray*}
\end{lemma}
\begin{lemma}\label{lemma2.2}
For $x\in\left[0,\infty\right)$, and with the help of Lemma \ref{lemma2.1}, we have
\begin{eqnarray*}
\tilde{\mathcal{B}}_n\left(1;x\right)&=& 1\\
\tilde{\mathcal{B}}_n\left(s;x\right)&=& \frac{\xi^{\prime}\left(p(x)\right)}{\xi\left(p(x)\right)}x+\frac{1}{n}\left(\frac{\mathcal{S}^{\prime}(1)}{\mathcal{S}(1)}
+\mathcal{V}^{\prime}(1)\frac{\xi^{\prime}\left(p(x)\right)}{\xi\left(p(x)\right)}\right)\\
\tilde{\mathcal{B}}_n\left(s^2;x\right)&=&\frac{n}{n-1}\frac{\xi^{\prime\prime}\left(p(x)\right)}{\xi\left(p(x)\right)}x^2+ \frac{x}{n-1}\left[2 \mathcal{V}^{\prime}(1)\frac{\xi^{\prime\prime}\left(p(x)\right)}{\xi\left(p(x)\right)}+\left(2+2\frac{\mathcal{S}^{\prime}(1)}{\mathcal{S}(1)}\right.\right.\\
&&\left.\left.+\mathcal{U}^{\prime\prime}(1)\right)\frac{\xi^{\prime}\left(p(x)\right)}
{\xi\left(p(x)\right)} \right]+\frac{1}{n(n-1)} \left[\left(\mathcal{V}^{\prime\prime}(1)\right)^2\frac{\xi^{\prime\prime}\left(p(x)\right)}{\xi\left(p(x)\right)}\right.\\
&&\left.+\left(2\frac{\mathcal{S}^{\prime}(1)}{\mathcal{S}(1)}\mathcal{V}^{\prime}(1)+2\mathcal{V}^{\prime}(1)+\mathcal{V}^{\prime\prime}(1) \right)\frac{\xi^{\prime}\left(p(x)\right)}{\xi\left(p(x)\right)}+\frac{2\mathcal{S}^{\prime}(1)+\mathcal{S}^{\prime\prime}(1)}{\mathcal{S}(1)}\right].
\end{eqnarray*}
\end{lemma}
\begin{proof}
By simple calculations, we get
\begin{eqnarray*}
\tilde{\mathcal{B}}_n\left(s^2;x\right)&=&\frac{1}{n-1}\left[n\mathcal{B}_n\left(s^2;x\right)+\mathcal{B}_n\left(s;x\right) \right].
\end{eqnarray*} Now, by simple computations and  using the Lemma \ref{lemma2.1}, required result follows.
\end{proof}
\begin{lemma}\label{lemma2.3}
Using Lemma \ref{lemma2.1} and Lemma \ref{lemma2.2}, central moments of the operators (\ref{eq1.2}) are given by,
\begin{eqnarray*}
\tilde{\mathcal{B}}_n\left((s-x);x\right)&=&\left[\frac{\xi^{\prime}\left(p(x)\right)}{\xi\left(p(x)\right)}-1\right]x+\frac{1}{n}\left(\frac{\mathcal{S}^{\prime}(1)}{\mathcal{S}(1)}
+\mathcal{V}^{\prime}(1)\frac{\xi^{\prime}\left(p(x)\right)}{\xi\left(p(x)\right)}\right)\\
\tilde{\mathcal{B}}_n\left((s-x)^2;x\right)&=&\left[\frac{n}{n-1}\frac{\xi^{\prime\prime}\left(p(x)\right)}{\xi\left(p(x)\right)}-2\frac{\xi^{\prime}
\left(p(x)\right)}{\xi\left(p(x)\right)}+1 \right]x^2\\
&&+\left[\frac{2}{n-1}\mathcal{V}^{\prime}(1)\frac{\xi^{\prime\prime}\left(p(x)\right)}{\xi\left(p(x)\right)}+\frac{1}{n-1}\left(2+\mathcal{U}^{\prime\prime}(1)\right.\right.\\
&&\left.\left.+\frac{2
\mathcal{S}^{\prime}(1)}{\mathcal{S}(1)} \right)\frac{\xi^{\prime}\left(p(x)\right)}{\xi\left(p(x)\right)} -\frac{2}{n}\left(\mathcal{V}^{\prime}(1)\frac{\xi^{\prime}\left(p(x)\right)}{\xi\left(p(x)\right)}+\frac{\mathcal{S}^{\prime}(1)}{\mathcal{S}(1)} \right) \right]x\\
&&+\frac{1}{n(n-1)} \left[\left(\mathcal{V}^{\prime}(1) \right)^{2}\frac{\xi^{\prime\prime}\left(p(x)\right)}{\xi\left(p(x)\right)}+\left( 2\mathcal{V}^{\prime}(1)\frac{\mathcal{S}^{\prime}(1)}{\mathcal{S}(1)}\right.\right.\\
&&\left.\left.+2 \mathcal{V}^{\prime}(1)+\mathcal{V}^{\prime\prime}(1) \right)\frac{\xi^{\prime}\left(p(x)\right)}{\xi\left(p(x)\right)}+\frac{2\mathcal{S}^{\prime}(1)+\mathcal{S}^{\prime\prime}(1)}{\mathcal{S}(1)} \right].
\end{eqnarray*}
\end{lemma}
\begin{proof}
Simple computations and the Lemma \ref{lemma2.2}, leads the required proof.
\end{proof}
Now, to study the convergence properties of the operators (\ref{eq1.2}), we assume that
$\displaystyle\lim_{t\rightarrow \infty}\frac{\xi^{\prime}(t)}{\xi(t)}=1$ and $\displaystyle\lim_{t\rightarrow \infty}\frac{\xi^{\prime\prime}(t)}{\xi(t)}=1$ that are point wise valid on the analytic functions $\mathcal{S}(s), \mathcal{T}(s),  \mathcal{U}(s)$ and $\mathcal{V}(s)$. Also, we assume the following considerations:\\
\begin{eqnarray*}
\displaystyle\lim_{n \rightarrow \infty}n \left[\frac{\xi^{\prime}\left(p(x)\right)}{\xi\left(p(x)\right)}-1\right]= \ell_{1}(x),
\end{eqnarray*}
\begin{eqnarray*}
\displaystyle\lim_{n \rightarrow \infty}n \left[\frac{n}{n-1}\frac{\xi^{\prime\prime}\left(p(x)\right)}{\xi\left(p(x)\right)}-n \frac{\xi^{\prime}\left(p(x)\right)}{\xi\left(p(x)\right)}+1\right]= \ell_{2}(x).
\end{eqnarray*}
 In the next Lemma, we discuss the limiting cases of the operators (\ref{eq1.2}), by applying the above assumptions on the Lemma \ref{lemma2.3}.
\begin{lemma}\label{lemma2.4}
For the operators (\ref{eq1.2}), we may write:
\begin{eqnarray*}
 \displaystyle\lim_{n \rightarrow \infty}n \tilde{\mathcal{B}}_n\left((s-x);x\right)= \ell_{1}(x)x + \frac{\mathcal{S}^{\prime}(1)}{\mathcal{S}(1)}+\mathcal{V}^{\prime}(1),
 \end{eqnarray*}
 \begin{eqnarray*}
\displaystyle\lim_{n \rightarrow \infty}n\tilde{\mathcal{B}}_n\left((s-x)^2;x\right)&=&\ell_{2}(x) x^2 + x\left(2+\mathcal{U}^{\prime\prime}(1)\right)\\
&=&\eta_{1}(x)(\text{say}).
\end{eqnarray*}
Moreover, for $n\in \mathds{N}$ and some constant $c$, we have
\begin{eqnarray*}
\tilde{\mathcal{B}}_n\left((s-x)^2;x\right)&\leq& \frac{\ell_{2}(x) x^2 + x\left(2+\mathcal{U}^{\prime\prime}(1)\right)}{n}\\
&\leq& \frac{\ell_{2}(x) x^2 + c x}{n}.
\end{eqnarray*} Since the limiting value $\ell_{2}(x)$ exists finitely, therefore for all $ x\in\left[0,\infty\right) $, there exists $\mathcal{M}>0$, such that $ \ell_{2}(x)\leq \mathcal{M}$, taking it into the account, we may write:
\begin{eqnarray*}
\tilde{\mathcal{B}}_n\left((s-x)^2;x\right)\leq \frac{\mathcal{M}.x(x+1)}{n}.
\end{eqnarray*}
\end{lemma}

\begin{thm}\label{theorem2.1}
For any compact subset $S$ of $\left[0,\infty\right)$ and a continuous function $f$ defined on $\left[0,\infty\right)$, the sequence $\tilde{\mathcal{B}}_n\left(f;x\right)_{n\geq1}$, converges uniformly to the function $f$ on $S$.
\end{thm}
\begin{proof}
By simple computations and using the Lemma \ref{lemma2.2} we see that, For $i=0, 1, 2,$ the sequence $\tilde{\mathcal{B}}_n\left(t^i;x\right)_{n\geq1}$, converges uniformly on $S$ to the function
$e_{i}(x)=x^i$. Hence, by using Bohman-Korovkin's theorem \cite{adg}, the required result follows.
\end{proof}
\section{Key Results}
Let's consider the Lipchitz-type space explained by O. Sz\'{a}sz \cite{szasz}to build the convergence of the operators (\ref{eq1.2}) in this space. For $0<r\leq 1, x\in(0,\infty)$ and $s\in \left[0,\infty\right)$, we have
\begin{eqnarray*}
Lip_{\mathcal{K}}^{*}:=\{ f\in C\left[0,\infty \right):\left|f(s)-f(x) \right|\}\leq \mathcal{K}_{f}\frac{\left|s-x\right|^{r}}{(s+x)^{\frac{r}{2}}},
\end{eqnarray*}where the constant $\mathcal{K}_{f}$ depends on $f$. In the following theorem, we establish a relation towards the rate of convergence of the operators $\tilde{\mathcal{B}}_n\left(f;x\right)$ in the Lipchitz-type space.
\begin{thm}{(Point-wise convergence) }
Let $f$ be a function in a Lipchitz class $Lip_{\mathcal{K}}^{*}$ and $r\in\left(0,1\right]$. Then,
\begin{eqnarray*}
 \left|\tilde{\mathcal{B}}_n\left(f;x\right)-f(x)\right|\leq \frac{\mathcal{K}}{x^{\frac{r}{2}}}\left(\tilde{\mathcal{B}}_n\left((s-x)^2;x\right)\right)^{\frac{r}{2}}.
\end{eqnarray*}
\end{thm}
\begin{proof}
\begin{eqnarray*}
\left|\tilde{\mathcal{B}}_n\left(f;x\right)-f(x)\right|&\leq&\int_{0}^{\infty}\underline{\mathcal{V}}_{n}(x,s)\left|f(s)-f(x)\right|ds,
\end{eqnarray*}on applying H\"{o}lder's inequality with $p=\frac{2}{r}$ and $q=\frac{2-r}{2}$ and Lemma \ref{lemma2.3}, we have
\begin{eqnarray*}
\left|\tilde{\mathcal{B}}_n\left(f;x\right)-f(x)\right|&\leq&\left(\int_{0}^{\infty}\underline{\mathcal{V}}_{n}(x,s)\left(\left|f(s)-f(x)\right|^{\frac{2}{r}}
\right)ds\right)^{\frac{r}{2}} \left(\int_{0}^{\infty}\underline{\mathcal{V}}_{n}(x,s)ds\right)^{\frac{2-r}{2}}\\
&\leq&\left(\int_{0}^{\infty}\underline{\mathcal{V}}_{n}(x,s)\left(\left|f(s)-f(x)\right|^{\frac{2}{r}}
\right)ds\right)^{\frac{r}{2}}\\
&\leq& \mathcal{K} \left(\int_{0}^{\infty}\underline{\mathcal{V}}_{n}(x,s)\left(\frac{(s-x)^2}{s+x}\right)ds\right)^{\frac{r}{2}}\\
&\leq& \frac{\mathcal{K}}{x^{\frac{r}{2}}}\left(\tilde{\mathcal{B}}_n\left((s-x)^2;x\right)\right)^{\frac{r}{2}}.
\end{eqnarray*}
\end{proof}

\begin{thm}\label{theorem3.3}
Let $\mathcal{C}_{B}\left[0,\infty\right)$ be the class of all real-valued continuous and bounded functions and for  $f$. Then, for  $x\in\left[0,\infty\right]$ and $\delta >0$, we have
\begin{eqnarray*}
\left|\tilde{\mathcal{B}}_n\left(f;x\right)\right|\leq 2\varpi\left(f;\sqrt{\tilde{\mathcal{B}}_n\left((s-x)^{2};x\right)} \right),
\end{eqnarray*}where,
 the modulus of continuity $\varpi(f;\delta)$ is given by
\begin{eqnarray*}
\varpi(f;\delta):=\displaystyle\sup_{|x-y|<\delta}\sup_{x,y\in\left[0,\infty\right)}\left|f(x)-f(y)\right|.
\end{eqnarray*}
\end{thm}
\begin{proof}For $f\in \mathcal{C}_{B}\left[0,\infty\right)$, we have
\begin{eqnarray}\label{eq3.1}
\left|\tilde{\mathcal{B}}_n\left(f;x\right)-f(x)\right|&=& \left|\frac{1}{\mathcal{S}(1).\xi\left(p(x)\right)}\sum_{j=0}^{\infty}\frac{\Theta_j(nx)}{B(j,n+1)}
\right.\nonumber\\
&&\left.\times\int_{0}^{\infty}\frac{s^{j-1}}{(1+s)^{n+j+1}}\left(f(s)-f(x)\right)ds\right|\nonumber\\
&\leq& \frac{1}{\mathcal{S}(1).\xi\left(p(x)\right)}\sum_{j=0}^{\infty}\frac{\Theta_j(nx)}{B(j,n+1)}\nonumber\\
&&\times \int_{0}^{\infty}\frac{s^{j-1}}{(1+s)^{n+j+1}}\left|f(s)-f(x)\right|ds\nonumber\\
&\leq& \frac{1}{\mathcal{S}(1).\xi\left(p(x)\right)}\sum_{j=0}^{\infty}\frac{\Theta_j(nx)}{B(j,n+1)}\nonumber\\
&&\times \int_{0}^{\infty}\frac{s^{j-1}}{(1+s)^{n+j+1}}\left(1+\frac{1}{\delta}\left|s-x\right|\right)\varpi(f;\delta)ds\nonumber\\
&\leq&\left(1+\frac{1}{\delta} \frac{1}{\mathcal{S}(1).\xi\left(p(x)\right)}\sum_{j=0}^{\infty}\frac{\Theta_j(nx)}{B(j,n+1)}\right.\nonumber\\
&&\left. \times \int_{0}^{\infty}\frac{s^{j-1}}{(1+s)^{n+j+1}}\left|s-x\right|ds\right) \varpi(f;\delta),
\end{eqnarray}on the account of Cauchy-Schwarz inequality and Lemma\ref{lemma2.4}, we get
\begin{eqnarray}\label{eq3.2}
\sum_{j=0}^{\infty}\frac{\Theta_j(nx)}{B(j,n+1)}
\int_{0}^{\infty}\frac{s^{j-1}}{(1+s)^{n+j+1}}\left|s-x\right|ds\hspace{2cm}\nonumber\\
 \leq\left(\sum_{j=0}^{\infty}\frac{\Theta_j(nx)}{B(j,n+1)}
\int_{0}^{\infty}\frac{s^{j-1}}{(1+s)^{n+j+1}}\left(s-x\right)^{2}ds\right)^{\frac{1}{2}}\nonumber\\
\times \left(\sum_{j=0}^{\infty}\frac{\Theta_j(nx)}{B(j,n+1)}
\int_{0}^{\infty}\frac{s^{j-1}}{(1+s)^{n+j+1}}ds\right)^{\frac{1}{2}}.
\end{eqnarray}From (\ref{eq3.1}) and (\ref{eq3.2}), we may write
\begin{eqnarray*}
\left|\tilde{\mathcal{B}}_n\left(f;x\right)-f(x)\right|&\leq& \left( 1+\frac{1}{\delta} \left(\tilde{\mathcal{B}}_n\left((1;x\right)\right)^{\frac{1}{2}}\left(\tilde{\mathcal{B}}_n\left((s-x)^{2};x\right)\right)^{\frac{1}{2}}\right)\varpi(f;\delta),
\end{eqnarray*} considering $\delta=\left(\tilde{\mathcal{B}}_n\left((s-x)^{2};x\right)\right)^{\frac{1}{2}}$ and using the above inequality, we get the required result.
\end{proof}

Now, we assess the rate of convergence by using the Ditzian-Totik modulus of smoothness $\varpi_{\chi^\gamma}(f,\delta)$ and Peetre's $K-$functional $K_{\chi^\gamma}(f,\delta), 0\leq\gamma\leq1.$ For $f\in C_B\left[0,\infty\right)$ and $\chi(x)=\sqrt{x(1+x)},$ the Ditzian-Totik modulus of smoothness is explained as:
\begin{eqnarray*}
\varpi_{\chi^\gamma}(f,\delta)=\sup_{0\leq i\leq\delta} \sup_{x\pm\frac{i\chi^{\gamma}(x)}{2}\in \left[0,\infty\right)}\left|f\left(x+\frac{i\chi^{\gamma}(x)}{2}\right)-f\left(x-\frac{i\chi^{\gamma}(x)}{2}\right)\right|,
\end{eqnarray*} and the Peetre's $K-$functional is defined as:
\begin{eqnarray*}
K_{\chi^\gamma}(f,\delta)=
\inf_{\varphi\in W_{\gamma}}\{\|f-\varphi\|-\delta\|\chi^{\gamma}\varphi^{\prime}\|\},
\end{eqnarray*}where $W_{\gamma}$ is a subspace of all real valued functions defined on $\left[0,\infty\right)$, and  $\varphi \in W_{\gamma}$ which is locally absolutely continuous with norm $\|f^\gamma \varphi^{\prime}\|\leq \infty.$ In [\cite{Ditz}, Theorem $2.1.1$], there exists a constant $ \mathcal{D}\geq0 $ such that
\begin{eqnarray}\label{eq3.3}
\mathcal{D}^{-1}\varpi_{\chi^\gamma}(f,\delta)\leq K_{\chi^\gamma}(f,\delta)\leq \mathcal{D}\varpi_{\chi^\gamma}(f,\delta).
\end{eqnarray}

\begin{thm}\label{thm3.4}
For $f \in C_{B}\left[0,\infty\right) $ then, we have
\begin{eqnarray*}
\left|\mathcal{B}_{n}(f;x)-f(x)\right|\leq \mathcal{D}\varpi_{\chi^s}\left(f;\chi^{1-\gamma}(x)\frac{1}{\sqrt{n}}\right).
\end{eqnarray*}
\end{thm}
\begin{proof}
For $\varphi\in W_{\gamma}$, first we write the following
\begin{eqnarray*}
 \varphi(s)=\varphi(x)+\int_{x}^{s}\varphi^{\prime}(s)ds.
\end{eqnarray*}
Applying $\mathcal{B}_{n}(.;x)$ and using H\"{o}lder's inequality, we have
\begin{eqnarray}\label{eq3.4}
\left|\tilde{\mathcal{B}}_{n}\left(\varphi(s);x\right)-\varphi(x)\right|&\leq&
\tilde{\mathcal{B}}_{n}\left(\int_{x}^{s}\left|\varphi^{\prime}(u)\right|du;x\right)\nonumber\\
&\leq&\|f^\gamma \varphi^{\prime}\|\tilde{\mathcal{B}}_{n}\left(\left|\int_{x}^{s}\frac{du}{\chi^{\gamma}(u)}\right| ;x \right)\nonumber\\
&\leq&\|f^\gamma \varphi^{\prime}\|\tilde{\mathcal{B}}_{n}\left(\left|s-x\right|^{1-\gamma}\left|\int_{x}^{s}\frac{du}{\chi(u)}\right|^\gamma ;x \right).
\end{eqnarray}
Let $I=\left|\int_{x}^{s}\frac{du}{\chi(u)}\right|$, now first we simplify expression $I$,
\begin{eqnarray}\label{eq3.5}
\left|\int_{x}^{s}\frac{du}{\chi(u)}\right|&\leq&\left|\int_{x}^{s}\frac{du}{\sqrt{u}}\right|\left(\frac{1}{1+x}+\frac{1}{1+s}\right)\\
&\leq&2\left|\sqrt{s}-\sqrt{x}\right|\left(\frac{1}{1+x}+\frac{1}{1+s}\right)\nonumber\\
&\leq&2\frac{\left|s-x\right|}{\sqrt{s}+\sqrt{x}}\left(\frac{1}{1+x}+\frac{1}{1+s}\right).\nonumber
\end{eqnarray}
Now, we use the inequality $\left|p+q\right|^\gamma\leq\left|p\right|^\gamma+\left|q\right|^\gamma,  0\leq\gamma\leq1, $ then from (\ref{eq3.5}), we get
\begin{eqnarray}\label{eq3.6}
\left|\int_{x}^{s}\frac{du}{\chi(u)}\right|^\gamma\leq2^\gamma\frac{\left|s-x\right|^\gamma}{x^{\frac{\gamma}{2}}}\left(\frac{1}{(1+x)^{\frac{\gamma}{2}}}
+\frac{1}{(1+s)^{\frac{\gamma}{2}}}\right).
\end{eqnarray}From (\ref{eq3.4}) and (\ref{eq3.6}) and using Cauchy inequality, we get
\begin{eqnarray}\label{eq3.7}
\left|\tilde{\mathcal{B}}_{n}(\varphi(s);x)-\varphi(x)\right|&\leq&\frac{2^\gamma\|\chi^\gamma \varphi^{\prime}\|}{x^{\frac{\gamma}{2}}}\tilde{\mathcal{B}}_{n}\left(\left|s-x\right|\left(\frac{1}{(1+x)^{\frac{\gamma}{2}}}
+\frac{1}{(1+s)^{\frac{\gamma}{2}}}\right);x\right)\nonumber\\
&=&\frac{2^\gamma\|\chi^\gamma \varphi^{\prime}\|}{x^{\frac{\gamma}{2}}}\left(\frac{1}{(1+ x)^{\frac{\gamma}{2}}}\left(\tilde{\mathcal{B}}_{n}\left((s-x)^2;x \right)\right)^{\frac{1}{2}}\right.\nonumber\\
&&\left.+\left(\tilde{\mathcal{B}}_{n}\left((s-x)^2;x \right)\right)^{\frac{1}{2}}.\left(\tilde{\mathcal{B}}_{n}\left((1+s)^{-\gamma};x \right)\right)^{\frac{1}{2}}\right).
\end{eqnarray}
From Lemma \ref{lemma2.2}, we may write
\begin{eqnarray}\label{eq3.8}
\left(\tilde{\mathcal{B}}_{n}\left((s-x)^2;x \right)\right)^{\frac{1}{2}}\leq \sqrt{\frac{\mathcal{M}}{n}}.\chi(x),
\end{eqnarray}
where $\chi(x)=\sqrt{x(1+x)}.$\\
For $x\in\left[0,\infty\right), \tilde{\mathcal{B}}_{n}\left((1+s)^{-\gamma};x\right)\rightarrow(1+x)^{-\gamma}$ as $n\rightarrow\infty.$ Thus for $\varepsilon >0,$ we find a number $n_{0}\in\mathbb{N}$ such that
\begin{eqnarray*}
\tilde{\mathcal{B}}_{n}\left((1+s)^{-\gamma};x\right)\leq (1+x)^{-\gamma}+\varepsilon,\text{ for all } n\geq n_{0}.
\end{eqnarray*}
By choosing $\varepsilon=(1+x)^{-\gamma} $, we obtain
\begin{eqnarray}\label{eq3.9}
\tilde{\mathcal{B}}_{n}\left((1+s)^{-\gamma};x\right)\leq 2(1+x)^{-\gamma}, \text{ for all } n\geq n_{0}.
\end{eqnarray}
From (\ref{eq3.7})-(\ref{eq3.9}), we have
\begin{eqnarray}\label{eq3.10}
\left|\tilde{\mathcal{B}}_{n}(\varphi(\lambda);x)-\varphi(x)\right|&\leq& 2^\gamma\|\chi^\gamma \varphi^{\prime}\| \left(\tilde{\mathcal{B}}_{n}\left((s-x)^2;x \right)\right)^{\frac{1}{2}}\left\{\chi^{-\gamma}(x)\right.\nonumber\\
&&\left.+ x^{-\frac{\gamma}{2}}\left(\tilde{\mathcal{B}}_{n}\left((1+s)^{-\gamma};x \right)\right)^{\frac{1}{2}}\right\}\nonumber\\
&\leq&2^\gamma\|\chi^\gamma \varphi^{\prime}\|\sqrt{\frac{\mathcal{M}}{n}}\chi(x)\left(\chi^{-\gamma}(x)+\sqrt{2}x^{-\frac{\gamma}{2}}(1+x)^{-\frac{\gamma}{2}}\right)\nonumber\\
&\leq&2^\gamma(1+\sqrt{2})\|\chi^\gamma \varphi^{\prime}\|\sqrt{\frac{\mathcal{M}}{n}}\chi^{1-\gamma}(x),
\end{eqnarray}
now, we write
\begin{eqnarray}\label{eq3.11}
\left|\tilde{\mathcal{B}}_{n}(f(s);x)-f(x)\right|&\leq&\left|\tilde{\mathcal{B}}_{n}\left(f(s)-\varphi(s);x\right)\right|\nonumber\\
&&+\left|\tilde{\mathcal{B}}_{n}\left(\varphi(s);x\right)-\varphi(x)\right|+\left|\varphi(x)-f(x)\right|\nonumber\\
&\leq&2\|f-\varphi\|+\left|\tilde{\mathcal{B}}_{n}\left(\varphi(s);x\right)-\varphi(x)\right|.
\end{eqnarray}
From (\ref{eq3.10}) and (\ref{eq3.11}) and for sufficiently large $n$, we obtain
\begin{eqnarray}\label{eq3.12}
\left|\tilde{\mathcal{B}}_{n}\left(f(s);x\right)-f(x)\right|&\leq&2\|f-\varphi\|+2^\gamma\left(1+\sqrt{2}\right)\|\chi^\gamma \varphi^{\prime}\|\sqrt{\frac{\mathcal{M}}{n}}\chi^{1-\gamma}(x)\nonumber\\
&\leq& \mathcal{D}\left\{\|f-\varphi\|+\frac{\chi^{1-\gamma}(x)}{\sqrt{n}}\|\chi^\gamma \varphi^{\prime}\|\right\}\nonumber\\
&\leq& \mathcal{D}\varpi_{\chi^s}\left(f;\chi^{1-\gamma}(x)\frac{1}{\sqrt{n}}\right),
\end{eqnarray}
where $\mathcal{D}\geq\max\{2,2^{s}(1+\sqrt{2})\sqrt{\mathcal{M}}\}$. From (\ref{eq3.12}) the result is concluded.
\end{proof}
\section{Weighted Approximation}
First we define a class of functions $ C_B^{x^2}\left[0,\infty\right)=\{f\in C_B\left[0,\infty\right) \text{for each }x\in \left[0,\infty\right); \left|f(x)\right|\leq \mathcal{K}_{f}(1+x^2)\}$. Let $C_B^{*}\left[0,\infty\right)$ be a subspace of $C_B^{x^2}\left[0,\infty\right)$ consisting all the real functions existing the following limit
$\displaystyle\lim_{n\rightarrow \infty} \frac{\left|f(x)\right|}{1+x^2} $ and norm in the subspace $ C_B^{*}\left[0,\infty\right)$ is defined by $\|f\|_{x^2}=\displaystyle\sup_{x\in\left[0,\infty\right)}\frac{f(x)}{1+x^2}$. The weighted modulus of continuity for any  $f\in C_B^{*}\left[0,\infty\right) $ and $\delta>0 $, by \cite{Ispir} is explained as:
\begin{eqnarray}\label{eq4.1}
\vartheta\left(f,\delta\right)=\displaystyle\sup_{x\in\left[0,\infty\right), 0< h\leq \delta}\frac{f(x+h)-f(x)}{1+(x+h)^2},
\end{eqnarray} where $\vartheta\left(f,\delta\right)$ holds the following properties:
  \begin{enumerate}
  \item $\vartheta\left(f,\delta\right)$ is an increasing function of $\delta$, and  $\displaystyle\lim_{n\rightarrow \infty}\vartheta\left(f,\delta\right)=0$.
  \item For each $m\in \mathds{N}, \vartheta\left(f,k\delta\right)\leq k \vartheta\left(f,\delta\right)$ and for each $\alpha\in(0,\infty), \vartheta\left(f,\alpha\delta\right)\leq(1+\alpha)\vartheta\left(f,\delta\right)$.
\end{enumerate}
\begin{thm}\label{theorem4.1} For each real function $f\in C_B^{*}\left[0,\infty\right),$ we have
	\begin{eqnarray*}
		\lim_{n\rightarrow{\infty}}\|\tilde{\mathcal{B}}_n\left(f;.\right)-f\|_{x^2}=0.
	\end{eqnarray*}
\end{thm}

\begin{proof}
	To prove this theorem, it is acceptable to verify the following conditions,
	\begin{eqnarray}\label{eq4.2}
	\lim_{n\rightarrow{\infty}}\|\tilde{\mathcal{B}}_n(t^r;x)-x^r\|_{x^2}=0,\,\,\, r=0,1,2.
	\end{eqnarray}
	Since $\tilde{\mathcal{B}}_n(1;x)=1$, therefore for $r=0\,$  holds.\\
	Using Lemma \ref{lemma2.3}, we have
	\begin{eqnarray*}
		\|\tilde{\mathcal{B}}_n(t;x)-x\|_{x^2}&=&\sup_{x\in[0,\infty)}\frac{|\tilde{\mathcal{B}}_n(t;x)-x|}{1+x^2}\\
		&\leq&\bigg|\frac{\xi^{\prime}\left(p(x)\right)}{\xi\left(p(x)\right)}x+\frac{1}{n}\left(\frac{\mathcal{S}^{\prime}(1)}{\mathcal{S}(1)}+
           \mathcal{V}^{\prime}(1)\frac{\xi^{\prime}\left(p(x)\right)}{\xi\left(p(x)\right)}\right)-x\bigg|\\
           && \times \sup_{x\in[0,\infty)}\frac{1}{1+x^2}.
\end{eqnarray*}
As $n \rightarrow \infty$ and using $\displaystyle\lim_{t\rightarrow \infty}\frac{\xi^{\prime}(t)}{\xi(t)}=1$ and $\displaystyle\lim_{t\rightarrow \infty}\frac{\xi^{\prime\prime}(t)}{\xi(t)}=1$,  the required inequality of the theorem holds good for $r=1$.
In the similar way by using the Lemma \ref{lemma2.3}, we may obtain $\|\tilde{\mathcal{B}}_n(t^2;x)-x^2\|_{x^2}=0.$	
It follows that the theorem holds for $r=2$ as ${n\rightarrow{\infty}}.$ Thus, on applying Korovkin's theorem required result follows.
\end{proof}
\begin{cor} Let $\alpha>0$ and $f\in C_B^{*}\left[0,\infty\right)$, then
	\begin{eqnarray*} \lim_{n\rightarrow{\infty}}\sup_{x\in[0,\infty)}\frac{|\tilde{\mathcal{B}}_n(f;x)-f(x)|}{(1+x^2)^{1+\alpha}}=0.
	\end{eqnarray*}
\end{cor}
\begin{proof}
For any fixed $\epsilon_0>0$, we have
\begin{align*}
\sup_{x\in[0,\infty)}\frac{|\tilde{\mathcal{B}}_n(f;x)-f(x)|}{(1+x^2)^{1+\alpha}}&\leq \sup_{x\leq \epsilon_0}\frac{|\tilde{\mathcal{B}}_n(f;x)-f(x)|}{(1+x^2)^{1+\alpha}}
		+\sup_{x\geq \epsilon_0}\frac{|\tilde{\mathcal{B}}_n(f;x)-f(x)|}{(1+x^2)^{1+\alpha}}\\
&\leq \|\tilde{\mathcal{B}}_n(f;.)-f\|_{C[0,\epsilon_0]}+\|f\|_{x^2}\sup_{x\geq \epsilon_0}\frac{|\tilde{\mathcal{B}}_n\left((1+t^2);x\right)|}{(1+x^2)^{1+\alpha}}\\
&\hspace{.5cm}+\sup_{x\geq \epsilon_0}\frac{|f(x)|}{(1+x^2)^{1+\alpha}}.	
\end{align*}
Using theorem (\ref{theorem4.1}) and for the fixed $\epsilon_0>0$ ( sufficiently large), we quickly seen that $\displaystyle\sup_{x\geq \epsilon_0}\frac{|\tilde{\mathcal{B}}_n\left((1+t^2); x\right)|}{(1+x^2)^{1+\alpha}}$ tends to zero as $n\rightarrow\infty.$ Above fact concludes the proof.
\end{proof}

\section{Rate of Convergence}
Now we discuss the convergence rate of the operators (\ref{eq1.2}) with a derivative of bounded variation. Let $ f$ be a continuous function taken from $DBV[0,\infty)$, which is the class of real-valued continuous functions with bounded derivative in each sub-intervals of $[0,\infty)$. It very well may be seen that, for all $f \in DBV[0,\infty)$, represented by
\begin{eqnarray*}
f(x)=f(c)+\int_c^x g(t)dt,\,\,\, x\geq c>0,
\end{eqnarray*} where $g(t)$ is a function of bounded variation on the class of finite sub-intervals of $[0,\infty).$\\

Now we define an auxiliary function $\psi_x$ by
\begin{eqnarray*}
\psi_x(t)=\begin{cases}
 \psi(t)-\psi(x^-)& \text{ if } 0\leq t<x\\
0& \text{ if } t=x\\
\psi(t)-\psi(x^+) & \text{ if } x<t<\infty.
\end{cases}
\end{eqnarray*}
\begin{lemma}{\label{lemma5.1}}
For every $x\in(0,\infty)$ and $ n $ to be very large. Then
\begin{enumerate}
  \item For $0\leq y<x$, we may write
\begin{eqnarray*}
\xi_n(x,y)=\int_0^y \underline{\mathcal{V}}_{n}(x,s)ds\leq \frac{C_{1}\left|\eta_{1}(x)\right|}{(x-y)^2}.
\end{eqnarray*}
  \item If $x<z<\infty$, then
\begin{eqnarray*}
1-\xi_n(x,z)=\int_z^\infty \underline{\mathcal{V}}_{n}(x,s)ds \leq \frac{C_{1} \left|\eta_{1}(x)\right|}{(z-x)^2}.
\end{eqnarray*}
\end{enumerate}
\end{lemma}
\begin{proof}
Using Lemma \ref{lemma2.3} and \ref{lemma2.4}, for a very large $n$ and $0\leq y<x$, we have
\begin{eqnarray*}
\xi_n(x,y)&=& \int_0^y \underline{\mathcal{V}}_{n}(x,s)ds\\
&\leq& \int_0^y \frac{(x-t)^2}{(x-y)^2}\underline{\mathcal{V}}_{n}(x,s)ds\\
&\leq& \frac{1}{(x-y)^2}\tilde{\mathcal{B}}_n((x-s)^2;x)\\
&\leq& \frac{C_{1}\left|\eta_{1}(x)\right|}{n(x-y)^2}.
\end{eqnarray*}In the similar way we can obtain second result.
\end{proof}
\begin{thm}\label{theorem5.1}(Bounded Variation)
Let $f \in DBV[0,\infty)$ then for all $x\in(0,\infty) $ and a very large $n$, we have
\begin{eqnarray*}
\left|\tilde{\mathcal{B}}_n(f;x)-f(x)\right|&\leq & \left[\left(\frac{\xi^{\prime}\left(p(x)\right)}{\xi\left(p(x)\right)}-1\right)x+\frac{1}{n}\left(\frac{\mathcal{S}^{\prime}(1)}{\mathcal{S}(1)}
+\mathcal{V}^{\prime}(1)\frac{\xi^{\prime}\left(p(x)\right)}{\xi\left(p(x)\right)}\right)\right]\\
&&\left|\frac{f^{\prime}(x^+)+f^{\prime}(x^-)}{2}\right|+ \sqrt{C_1|\eta_1(x)|}\left|\frac{f^{\prime}(x^+)+f^{\prime}(x^-)}{2}\right| \\
&&+\frac{C_1|\eta_{1}(x)|}{x}\sum_{j=1}^{[\sqrt{n}]}\left( \bigvee_{x-\frac{x}{j}}^{x}f^{\prime}_x\right)+ \frac{x}{\sqrt{n}} \left(\bigvee_{x-\frac{x}{\sqrt{n}}}^{x}f^{\prime}_x\right)\\
&& \times
\left(4 \mathcal{K}_f +\frac{\left(\mathcal{K}_f+\left|f(x)\right|\right)}{x^2}\right)C_1|\eta_1(x)|+ \left|f^{\prime}(x^+)\right|\\
&&\times\sqrt{\left|C_1\eta_1(x)\right|}+ \frac{C_1|\eta_1(x)|}{x^2}\left|f(2x)-xf^{\prime}(x^+)-f(x) \right|\\
&&+ \frac{x}{\sqrt{n}}\left(\bigvee_{x}^{x+\frac{x}{\sqrt{n}}}f^{\prime}_x \right)+\frac{C_1|\eta_1(x)|}{x}\sum_{j=1}^{[\sqrt{n}]} \left(\bigvee_{x}^{x+\frac{x}{j}}f^{\prime}_x \right).
\end{eqnarray*}
\end{thm}
\begin{proof}
Using the operators (\ref{eq1.2}) and for all $x$ lies on positive real line, we obtain
\begin{eqnarray}\label{eq5.1}
\tilde{\mathcal{B}}_n(f;x)-f(x)&=& \int_0^\infty \tilde{\mathcal{B}}_n(x,s)(f(s)-f(x))ds\nonumber\\
&=& \int_0^\infty \left(\tilde{\mathcal{B}}_n(x,s)\int_{x}^{s} f(v)dv\right)ds.
\end{eqnarray} For $f\in DBV[0,\infty)$, we have
\begin{eqnarray}\label{eq5.2}
f^{\prime}(v)&=&\frac{1}{2}(f^{\prime}(x^+)+f^{\prime}(x^-))+f_x^{\prime}(v)+\frac{1}{2}(f^{\prime}(x^+)-
f^{\prime}(x^-))sgn(x)\nonumber \\
&&+\delta_x(v)(f^{\prime}(v)-\frac{1}{2}(f^{\prime}(x^+)+f^{\prime}(x^-)),
\end{eqnarray}where
\begin{eqnarray*}
\delta_x(v)=
\begin{cases}
1, & v= x\\
0, & v\neq x.
\end{cases}
\end{eqnarray*}
For every $x\in[0,\infty), $ using (\ref{eq5.2}) and (\ref{eq1.3}), we get
\begin{eqnarray}\label{eq5.3}
\tilde{\mathcal{B}}_n(f;x)-f(x)&=&\int_{0}^{\infty}\underline{\mathcal{V}}_{n}(x,s)\left(f(s)-f(x)\right)ds\nonumber\\
&=&\int_{0}^{\infty}\underline{\mathcal{V}}_{n}(x,s)\left(\int_{s}^{x}f^{\prime}(u)du\right)ds\nonumber\\
&=& - \int_{0}^{x}\underline{\mathcal{V}}_{n}(x,s)\left(\int_{s}^{x}f^{\prime}(u)du\right)ds \nonumber\\
&&+ \int_{x}^{\infty}\underline{\mathcal{V}}_{n}(x,s)\left(\int_{s}^{x}f^{\prime}(u)du\right)ds.
\end{eqnarray}Now, we simplify the above expressions into two parts. Let
\begin{eqnarray*}
A_{1}=\int_{0}^{x}\underline{\mathcal{V}}_{n}(x,s)\left(\int_{s}^{x}f^{\prime}(u)du\right)ds,
 \end{eqnarray*} and
 \begin{eqnarray*}
 A_{2}=  \int_{x}^{\infty}\underline{\mathcal{V}}_{n}(x,s)\left(\int_{s}^{x}f^{\prime}(u)du\right)ds.
\end{eqnarray*} Since $\int_{x}^{s}\delta_{x}(s)ds=0 $, using (\ref{eq5.2}), we get
\begin{eqnarray}\label{eq5.4}
A_{1}&=&\int_{0}^{x}\left\{\int_{s}^{x}\left(\frac{1}{2}(f^{\prime}(x^+)+f^{\prime}(x^-))+f_x^{\prime}(u)\right.\right.\nonumber\\
&&\left.\left.+\frac{1}{2}(f^{\prime}(x^+)-
f^{\prime}(x^-))sgn(u-x)\right)du\right\}\underline{\mathcal{V}}_{n}(x,s)ds\nonumber\\
&=& \frac{1}{2}\left(f^{\prime}(x^+)+f^{\prime}(x^-)\right)\int_{0}^{x}(x-s)\underline{\mathcal{V}}_{n}(x,s)ds\nonumber\\
&&+\int_{0}^{x}\underline{\mathcal{V}}_{n}(x,s)
\left(\int_{s}^{x}f^{\prime}_x(u)du\right)ds\nonumber\\
&& -\frac{1}{2}\left(f^{\prime}(x^+)-f^{\prime}(x^-)\right)\int_{0}^{x}(x-s)\underline{\mathcal{V}}_{n}(x,s)ds.
\end{eqnarray} Similarly, we may write
\begin{eqnarray}\label{eq5.5}
A_{2}&=&\frac{1}{2}\left(f^{\prime}(x^+)+f^{\prime}(x^-)\right)\int_{x}^{\infty}(s-x)\underline{\mathcal{V}}_{n}(x,s)ds\nonumber\\
&&+\int_{x}^{\infty}\underline{\mathcal{V}}_{n}(x,s)
\left(\int_{x}^{s}f^{\prime}_x(u)du\right)ds\nonumber\\
&& +\frac{1}{2}\left(f^{\prime}(x^+)-f^{\prime}(x^-)\right)\int_{x}^{\infty}(s-x)\underline{\mathcal{V}}_{n}(x,s)ds.
\end{eqnarray}Now, from the relations (\ref{eq5.3})-(\ref{eq5.5}), we get
\begin{eqnarray*}
\tilde{\mathcal{B}}_n(f;x)-f(x)&=&\frac{1}{2}\left(f^{\prime}(x^+)+f^{\prime}(x^-)\right)\int_{0}^{\infty}(s-x)\underline{\mathcal{V}}_{n}(x,s)ds\\
&&+\frac{1}{2}\left(f^{\prime}(x^+)-f^{\prime}(x^-)\right)\int_{0}^{\infty}\left|x-s\right|\underline{\mathcal{V}}_{n}(x,s)ds\\
&&-\int_{0}^{x}\left(\int_{t}^{x}f^{\prime}_x(u)du\right)\underline{\mathcal{V}}_{n}(x,s)ds\nonumber\\
&&+\int_{x}^{\infty}\left(\int_{x}^{t}f^{\prime}_x(u)du\right)\underline{\mathcal{V}}_{n}(x,s)ds.
\end{eqnarray*}
Thus,
\begin{eqnarray}\label{eq5.6}
\left|\tilde{\mathcal{B}}_n(f;x)-f(x)\right|&\leq&\left|\frac{f^{\prime}(x^+)+f^{\prime}(x^-)}{2} \right|\left|\tilde{\mathcal{B}}_n(s-x;x\right|\nonumber\\
&&+\left|\frac{f^{\prime}(x^+)-f^{\prime}(x^-)}{2} \right|\tilde{\mathcal{B}}_n(\left|s-x\right|;x)\nonumber\\
&& +S_{n}(f^{\prime},x)+ T_{n}(f^{\prime},x).
\end{eqnarray}Where,
\begin{eqnarray*}
S_{n}(f^{\prime},x)=\left|\int_{0}^{x}\left(\int_{s}^{x}f^{\prime}_x(u)du\right)\underline{\mathcal{V}}_{n}(x,s)ds\right|,
\end{eqnarray*}  and
\begin{eqnarray*}
T_{n}(f^{\prime},x)=\left|\int_{x}^{\infty}\left(\int_{x}^{s}f^{\prime}_x(u)du\right)\underline{\mathcal{V}}_{n}(x,s)ds\right|.
 \end{eqnarray*}
 Now, using the Lemma \ref{lemma5.1}, and the properties of integration, we estimate $S_{n}(f^{\prime},x)$ and $T_{n}(f^{\prime},x)$.
\begin{eqnarray*}
 S_{n}(f^{\prime},x)=\int_{0}^{x}\left(\int_{s}^{x}f^{\prime}_x(u)du\right)\frac{\partial \xi_{n}(x,s)}{\partial s}ds = \int_{0}^{x}f^{\prime}_x(s)\xi_{n}(x,s)ds
\end{eqnarray*}
\begin{eqnarray*}
\left|S_{n}(f^{\prime},x)\right|= \int_{0}^{x}|f^{\prime}_x(s)|\xi_{n}(x,s)ds &\leq &\int_{0}^{x-\frac{x}{\sqrt{n}}}\left|f^{\prime}_x(s) \right|\xi_{n}(x,s)ds\\
&&+\int_{x-\frac{x}{\sqrt{n}}}^{x}\left|f^{\prime}_x(s) \right|\xi_{n}(x,s)ds.
\end{eqnarray*}Using the fact $f^{\prime}_x(s)=0 $, and $\xi_{n}(x,s)\leq 1, $ we may write
\begin{eqnarray}\label{eq5.7}
\int_{x-\frac{x}{\sqrt{n}}}^{x}\left|f^{\prime}_x(s) \right|\xi_{n}(x,s)ds&=&\int_{x-\frac{x}{\sqrt{n}}}^{x}\left|f^{\prime}_x(s) -f^{\prime}_x(x) \right|\xi_{n}(x,s)ds\nonumber\\
&\leq& \int_{x-\frac{x}{\sqrt{n}}}^{x}\left( \bigvee_{s}^{x}f^{\prime}_x\right) ds
\leq \left(\bigvee_{x-\frac{x}{\sqrt{n}}}^{x}f^{\prime}_x\right) \int_{x-\frac{x}{\sqrt{n}}}^{x}  ds\nonumber\\
 &=& \frac{x}{\sqrt{n}} \left(\bigvee_{x-\frac{x}{\sqrt{n}}}^{x}f^{\prime}_x\right).
\end{eqnarray} Taking $s=x-\frac{x}{u}$, and use the Lemma \ref{lemma5.1}, we obtain
\begin{eqnarray}\label{eq5.8}
\int_{0}^{x-\frac{x}{\sqrt{n}}}\left|f^{\prime}_x(s) \right|\xi_{n}(x,s)ds&\leq& C_1|\eta_{1}(x)|\int_{0}^{x-\frac{x}{\sqrt{n}}}\frac{\left|f^{\prime}_x(s) \right|}{(x-s)^{2}}ds\nonumber\\
&\leq& C_1|\eta_{1}(x)|\int_{0}^{x-\frac{x}{\sqrt{n}}}\left( \bigvee_{s}^{x}f^{\prime}_x\right)\frac{ds}{(x-s)^{2}}\nonumber\\
 &=& \frac{C_1|\eta_{1}(x)|}{x}\int_{1}^{\sqrt{n}}\left( \bigvee_{x-\frac{x}{u}}^{x}f^{\prime}_x\right)du\nonumber\\
&\leq& \frac{C_1|\eta_{1}(x)|}{x}\sum_{j=1}^{[\sqrt{n}]}\left( \bigvee_{x-\frac{x}{j}}^{x}f^{\prime}_x\right).
\end{eqnarray} From (\ref{eq5.7}) and (\ref{eq5.8}), we get
\begin{eqnarray}\label{eq5.9}
\left|S_{n}(f^{\prime},x)\right| \leq \frac{C_1|\eta_{1}(x)|}{x}\sum_{j=1}^{[\sqrt{n}]}\left( \bigvee_{x-\frac{x}{j}}^{x}f^{\prime}_x\right)+ \frac{x}{\sqrt{n}} \left(\bigvee_{x-\frac{x}{\sqrt{n}}}^{x}f^{\prime}_x\right).
\end{eqnarray} Now, using integration by parts and Lemma \ref{lemma5.1}, we get
\begin{eqnarray}\label{eq5.10}
T_{n}(f^{\prime},x) &\leq& \left|\int_x^{2x} \left(\int_{x}^{s}f_x^{\prime}(u)du\right) \frac{\partial }{\partial s} \left(1-\xi_n(x,s)ds \right)\right|\nonumber\\
&&+ \left|\int_{2x}^{\infty} \left(\int_{x}^{s}f_x^{\prime}(u)du\right) \underline{\mathcal{V}}_{n}(x,s)ds \right|\nonumber\\
&\leq&\left|\int_x^{2x}f_x^{\prime}(u)du\right|\left|1-\xi_n(x,2x)\right|+ \int_x^{2x}\left|f_x^{\prime}(s)\right|\left(1-\xi_n(x,s)\right)ds\nonumber\\
&& + \left|\int_{2x}^{\infty}\left(f(s)-f(x)\right)\underline{\mathcal{V}}_{n}(x,s)ds\right|\nonumber\\
&&+ |f^{\prime}(x^+)| \left|\int_{2x}^{\infty}(s-x)\underline{\mathcal{V}}_{n}(x,s)ds \right|.
\end{eqnarray}Here, we simplify the following integration into two parts,
 \begin{eqnarray}\label{eq5.11}
\int_x^{2x}\left|f_x^{\prime}(s)\right|\left(1-\xi_n(x,s)\right)ds &=& \int_{x}^{x+\frac{x}{\sqrt{n}}}\left|f_x^{\prime}(s)\right|\left(1-\xi_n(x,s)\right)ds \nonumber\\
&&+ \int_{x+\frac{x}{\sqrt{n}}}^{2x}\left|f_x^{\prime}(s)\right|\left(1-\xi_n(x,s)\right)ds\nonumber\\
&=& I_1+I_2.
\end{eqnarray}Using the Lemma \ref{lemma5.1},
\begin{eqnarray}\label{eq5.12}
I_1&=& \int_{x}^{x+\frac{x}{\sqrt{n}}}\left|f_x^{\prime}(s)-f_x^{\prime}(x)\right|\left(1-\xi_n(x,s)\right)ds \nonumber\\ &\leq&\int_{x}^{x+\frac{x}{\sqrt{n}}}\left(\bigvee_{x}^{x+\frac{x}{\sqrt{n}}}f^{\prime}_x \right)ds= \frac{x}{\sqrt{n}}\left(\bigvee_{x}^{x+\frac{x}{\sqrt{n}}}f^{\prime}_x \right).
\end{eqnarray}Again, using the Lemma \ref{lemma5.1} and put $s=x+\frac{x}{u}$, we get
\begin{eqnarray}\label{eq5.13}
I_2 &\leq& C_1|\eta_1(x)|\int_{x+\frac{x}{\sqrt{n}}}^{2x}\frac{1}{(s-x)^2}\left|f_x^{\prime}(s)-f_x^{\prime}(x)\right|ds\nonumber\\
& \leq& C_1|\eta_1(x)|
\int_{x+\frac{x}{\sqrt{n}}}^{2x}\frac{1}{(s-x)^2}\left(\bigvee_{x}^{s}f^{\prime}_x \right)ds\nonumber\\
&\leq&\frac{C_1|\eta_1(x)|}{x}\int_{1}^{\sqrt{n}}\left(\bigvee_{x}^{x+\frac{x}{u}}f^{\prime}_x \right)du \nonumber\\
&\leq& \frac{C_1|\eta_1(x)|}{x}\sum_{j=1}^{[\sqrt{n}]}
\left(\bigvee_{x}^{x+\frac{x}{j}}f^{\prime}_x \right).
\end{eqnarray}Using (\ref{eq5.11})-(\ref{eq5.13}), we get
\begin{eqnarray}\label{eq5.14}
\int_x^{2x}\left|f_x^{\prime}(s)\right|\left(1-\xi_n(x,s)\right)ds &=& \frac{C_1|\eta_1(x)|}{x}\sum_{j=1}^{[\sqrt{n}]}
\left(\bigvee_{x}^{x+\frac{x}{j}}f^{\prime}_x \right)\nonumber\\
&&+\frac{x}{\sqrt{n}}\left(\bigvee_{x}^{x+\frac{x}{\sqrt{n}}}f^{\prime}_x \right).
\end{eqnarray}Now, from (\ref{eq5.10}) and (\ref{eq5.14}) and the Cauchy-Schwarz inequality taking into the account, we get
\begin{eqnarray}\label{eq5.15}
T_{n}(f^{\prime},x) &\leq& \mathcal{K}_{f}\int_{2x}^{\infty}(s^2+1)\underline{\mathcal{V}}_{n}(x,s)ds+|f(x)|\int_{2x}^{\infty}\underline{\mathcal{V}}_{n}(x,s)ds\nonumber\\
&& + \left|f^{\prime}(x^+)\right|\sqrt{\left|C_1\eta_1(x)\right|}+ \frac{C_1|\eta_1(x)|}{x^2}\left|f(2x)-xf^{\prime}(x^+)-f(x) \right|\nonumber\\
&&+\frac{x}{\sqrt{n}}\left(\bigvee_{x}^{x+\frac{x}{\sqrt{n}}}f^{\prime}_x \right) +\frac{C_1|\eta_1(x)|}{x}\sum_{j=1}^{[\sqrt{n}]}
\left(\bigvee_{x}^{x+\frac{x}{j}}f^{\prime}_x \right).
\end{eqnarray} Since $x\leq s-x $, and  $ s\leq 2(s-x)$ when $s\geq 2x$, we have
\begin{eqnarray}\label{eq5.16}
\mathcal{K}_{f}\int_{2x}^{\infty}(s^2+1)\underline{\mathcal{V}}_{n}(x,s)ds+|f(x)|\int_{2x}^{\infty}\underline{\mathcal{V}}_{n}(x,s)ds\hspace{3cm}\nonumber\\
\leq \left(\mathcal{K}_f+\left|f(x)\right|\right)\int_{2x}^{\infty}\underline{\mathcal{V}}_{n}(x,s)ds+ 4\mathcal{K}_f \int_{2x}^{\infty}(s-x)^2\underline{\mathcal{V}}_{n}(x,s)ds\hspace{1cm}\nonumber\\
\leq \frac{\left(\mathcal{K}_f+\left|f(x)\right|\right)}{x^2}\int_{0}^{\infty}(s-x)^2\underline{\mathcal{V}}_{n}(x,s)ds+4\mathcal{K}_f \int_{0}^{\infty}(s-x)^2\underline{\mathcal{V}}_{n}(x,s)ds \nonumber\\
\leq \left(4 \mathcal{K}_f +\frac{\left(\mathcal{K}_f+\left|f(x)\right|\right)}{x^2}\right)C_1|\eta_1(x)|.\hspace{5cm}
\end{eqnarray} Using (\ref{eq5.15}) and (\ref{eq5.16}), we have
\begin{eqnarray}\label{eq5.17}
T_{n}(f^{\prime},x) &\leq& \left(4 \mathcal{K}_f +\frac{\left(\mathcal{K}_f+\left|f(x)\right|\right)}{x^2}\right)C_1|\eta_1(x)|+ \left|f^{\prime}(x^+)\right|\sqrt{\left|C_1\eta_1(x)\right|}\nonumber\\
&&+ \frac{C_1|\eta_1(x)|}{x^2}\left|f(2x)-xf^{\prime}(x^+)-f(x) \right|+ \frac{x}{\sqrt{n}}\left(\bigvee_{x}^{x+\frac{x}{\sqrt{n}}}f^{\prime}_x \right) \nonumber\\
&& +\frac{C_1|\eta_1(x)|}{x}\sum_{j=1}^{[\sqrt{n}]} \left(\bigvee_{x}^{x+\frac{x}{j}}f^{\prime}_x \right).
\end{eqnarray} using (\ref{eq5.6}),(\ref{eq5.9}) and (\ref{eq5.17}), we reach to the required result.
\end{proof}
\section{Another extension of sequence of operators $\mathcal{B}_n\left(f;x\right)$ }
Here, we define the Sz'asz-type Durrmeyer variant of the operators $\mathcal{B}_n\left(f;x\right)$ under the same conditions and limitations as those outlined in the article's introduction. For $f\in C_{\sigma}\left[0,\infty\right) $ and
$p(x)= n^2x^2\mathcal{T}(1)+ nx \mathcal{U}(1)+\mathcal{V}(1)$, we have
\begin{align*}
\mathcal{B}^{*}_n\left(f;x\right)= \frac{n}{\mathcal{S}(1).\xi\left(p(x)\right)}\sum_{j=0}^{\infty}\Theta_{j}(nx)\int_{0}^{\infty}e^{-ns}\frac{(ns)^{j}}{j!}f(s)ds.
\end{align*} In a manner similar to that described in this article, one can prove the convergence results for the operators $\mathcal{B}^{*}_n\left(f;x\right) $.

\end{document}